\documentclass[11pt,reqno]{amsart}

\usepackage{amsopn,amstext,amsbsy,amsmath,amscd,amsthm,amsfonts}

\usepackage{color}
\usepackage{amsfonts,amscd}
\usepackage{amssymb}
\usepackage{url}
\usepackage[english]{babel}

\theoremstyle{plain}
\newtheorem{theorem}                 {Theorem}      [section]
\newtheorem{conjecture}   [theorem]  {Conjecture}
\newtheorem{corollary}    [theorem]  {Corollary}
\newtheorem{lemma}        [theorem]  {Lemma}
\newtheorem{proposition}  [theorem]  {Proposition}

\theoremstyle{definition}

\newtheorem{remark}       [theorem]  {Remark}

\numberwithin{equation}{section}

\def \H{{\mathbb H}}

\def \rn{{\mathbb R}}

\def \F{\mathcal F}

\def \H{\mathcal H}

\def \V{\mathcal V}

\def\nab#1#2{\hbox{$\nabla$\kern -.3em\lower 1.0 ex
		\hbox{$#1$}\kern -.1 em {$#2$}}}

\def \lb#1#2{[#1,#2]}

\def \g{\mathfrak{g}}

\def \k{\mathfrak{k}}

\def \m{\mathfrak{m}}

\def \SLR#1{\text{\bf SL}_{#1}(\rn)}
\def \slr#1{\mathfrak{sl}_{#1}(\rn)}

\def \SO#1{\text{\bf SO}(#1)}
\def \so#1{\mathfrak{so}(#1)}

\def \SU#1{\text{\bf SU}(#1)}
\def \su#1{\mathfrak{su}(#1)}

\def \veps#1{\varepsilon_#1}

\def\jid(#1#2#3){\left[\left[#1,#2\right],#3\right] + \left[\left[#3,#1\right],#2\right] + \left[\left[#2,#3\right],#1\right]}

\numberwithin{equation}{section}
\allowdisplaybreaks

\def\jid(#1#2#3){\left[\left[#1,#2\right],#3\right] + \left[\left[#3,#1\right],#2\right] + \left[\left[#2,#3\right],#1\right]} 
\def\BV(#1#2){B^{\V}(#1,#2)}
\def\hproj(#1#2){\frac{1}{2}\,\H\,(\nab(#1#2) + \nab(#2#1))}
\def\BH(#1#2){B^{\H}(#1,#2)}
\def\vproj(#1#2){\frac{1}{2}\,\V\,(\nab(#1#2) + \nab(#2#1))}

\def\lieb(#1#2){\left[#1,#2\right]}

\def \jacobi#1#2#3{[[#1, #2], #3] + [[#3, #1], #2] + [[#2, #3], #1]}

\begin{document}

\title[Conformal Minimal Foliations on Semi-Riemannian Lie Groups]{Conformal Minimal Foliations on Semi-Riemannian Lie Groups}


\author{Elsa Ghandour}
\address{Mathematics, Faculty of Science\\
University of Lund\\
Box 118, Lund 221\\
Sweden}
\email{Elsa.Ghandour@math.lu.se}

\author{Sigmundur Gudmundsson}
\address{Mathematics, Faculty of Science\\
	University of Lund\\
	Box 118, Lund 221\\
	Sweden}
\email{Sigmundur.Gudmundsson@math.lu.se}

\author{Victor Ottosson}
\address{Mathematics, Faculty of Science\\
	University of Lund\\
	Box 118, Lund 221\\
	Sweden}
\email{Victor@Ottosson.ooo}

\begin{abstract}
We study left-invariant foliations $\F$ on semi-Riemannian Lie groups $G$ generated by a subgroup $K$. We are interested in such foliations which are conformal and with minimal leaves of codimension two.  We classify  such foliations $\F$ when the subgroup $K$ is one of the important $\SU 2$, $\SLR 2$, $\SU 2\times\SU 2$, $\SU 2\times\SLR 2$, $\SU 2\times\SO 2$, $\SLR 2\times\SO 2$. This way we construct new multi-dimensional families of Lie groups $G$ carrying such foliations in each case.  These foliations $\F$ produce local complex-valued harmonic morphisms on the corresponding Lie group $G$.
\end{abstract}



\maketitle

\section{Introduction}
\label{section-introduction}

Let $(G,g)$ be a semi-Riemannian Lie group equipped with a conformal foliation $\F$ generated by the left-translations of a subgroup $K$ of codimension two.  The foliation is particularly interesting if its leaves are minimal submanifolds, since it then generates complex-valued harmonic morphisms locally defined on the group $G$. The work \cite{Gha-Gud-Tur-1} is an investigation of this situation in the special {\it Riemannian} case when $K$ is one of the groups  $$\SU2\times\SU 2,\ \SU 2\times\SLR 2,\ \SU 2\times\SO 2,\ \SLR 2\times\SO 2.$$  The authors classify the simply connected Lie groups $G$ carrying such a conformal foliation $\F$ with minimal leaves of codimension two. In each case they obtain  multi-dimensional solutions to this interesting geometric problem.  Further they show that their following conjecture holds in these cases.

\begin{conjecture}\label{conjecture-Riemannian}
Let $(G,g)$ be a Riemannian Lie group with a subgroup $K$ generating a left-invariant conformal foliation $\F$ on $G$ of codimension two. If $K$ is semisimple then the foliation $\F$ is minimal.  If $K$ is semisimple and compact then $\F$ is totally geodesic.
\end{conjecture}

Here our principal aim is to extend the investigation to semi-Riemannian Lie groups $G$ and study the validity of Conjecture \ref{conjecture-Riemannian} for the groups $K$ found in Table \ref{table-subgroups-K}.  Our results show that the situation is {\it quiet different}, in the general semi-Riemannian case, from the special one studied in \cite{Gha-Gud-Tur-1}. We manufacture plenty of examples contradicting Conjecture \ref{conjecture-Riemannian} in those situations.  Our observations lead us to the following.


\begin{conjecture}\label{conjecture-semi-Riemannian}
Let $(G,g)$ be a semi-Riemannian Lie group with a subgroup $K$ generating a left-invariant conformal foliation $\F$ on $G$ of codimension two. If $K$ is semisimple then the foliation $\F$ is minimal.
\end{conjecture}

For semi-Riemannian geometry we recommend O'Neill's classical work \cite{ONe}.  Readers not familiar with harmonic morphisms are advised to consult the standard text \cite{Bai-Woo-book}, by Baird and Wood, \cite{Fug-1}, \cite{Fug-2}, \cite{Ish} and the regularly updated online bibliography \cite{Gud-bib}.
\vskip .1cm

It is well-known that the fibres of a horizontally conformal
map (resp.\ semi-Riemannian submersion) give rise to a conformal foliation
(resp.\ semi-Riemannian foliation). Conversely, the leaves of any
conformal foliation (resp.\ semi-Riemannian foliation) are
locally the fibres of a horizontally conformal map
(resp.\ semi-Riemannian submersion), see \cite{Bai-Woo-book}.

The next result is a semi-Riemannian version of Theorem 5.2 in  \cite{Bai-Eel}, by Baird and Eells. This gives the theory of harmonic morphisms, with values in a surface, a strong geometric flavour.

\begin{theorem}\label{theo:B-E}
	Let $\phi:(M^m,g)\to (N^2,h)$ be a horizontally conformal submersion from a semi-Riemannian manifold to a surface. Then $\phi$ is harmonic if and only if $\phi$ has minimal fibres.
\end{theorem}

\renewcommand{\arraystretch}{1.5}
\begin{table}
\begin{center}
\begin{tabular}{ c | c | c }
$K$ & compact & non-compact \\
\hline\hline
simple & $\SU 2$ & $\SLR 2$ \\
\hline
semisimple& $\SU 2 \times \SU 2$ & $\SU 2\times\SLR 2$ \\
\hline
non-semisimple & $\SU 2\times\SO 2$ & $\SLR 2\times\SO 2$ \\
\hline\hline
\end{tabular}
\end{center}
\bigskip
\caption{ }
\label{table-subgroups-K}
\end{table}
\renewcommand{\arraystretch}{1}

\section{Conformal Semi-Riemannian Foliations}
\label{section-the-general setup}

Let $(M,g)$ be a semi-Riemannian manifold, $\V$ be an integrable distribution on $M$ and denote by $\H$ its orthogonal complement distribution.
As customary, we also use $\V$ and $\H$ to denote the orthogonal projections onto the corresponding subbundles of $TM$
and denote by $\F$ the foliation tangent to $\V$. Then the second fundamental form for $\V$ is given by
$$B^\V(E,F)=\tfrac 12\,\H(\nabla_EF+\nabla_FE)=\H(\nabla_ EF)\qquad(E,F\in\V),$$
while the second fundamental form for $\H$ satisfies 
$$B^\H(E,F)=\tfrac{1}{2}\,\V(\nabla_EF+\nabla_FE)\qquad(E,F\in\H).$$
The foliation $\F$ tangent to $\V$ is said to be {\it conformal} if there exists a vector field $V\in \V$ such that $$B^\H=g\otimes V,$$ and $\F$ is said to be {\it semi-Riemannian} if $V=0$. Furthermore, $\F$ is said to be {\it minimal} if $\text{trace}\ B^\V=0$ and {\it totally geodesic} if $B^\V=0$. This is equivalent to the leaves of $\F$ being minimal and totally geodesic submanifolds of $M$, respectively.

\begin{proposition}\label{general-conformality}
Let $(M,g)$ be a semi-Riemannian manifold, $\V$ be an integrable distribution on $M$ of codimension two and $\H$ be its orthogonal complementary distribution. Then the foliation $\F$ tangent to $\V$ is conformal if and only if 
$$\veps X\cdot B^\H(X,X)-\veps Y\cdot B^\H(Y,Y)=0\ \ \text{and}\ \ B^\H(X,Y)=0,$$
for any local orthonormal frame $\{X,Y\}$ for the horizontal distribution $\H$.  If $\F$ is conformal then it is semi-Riemannian if and only if 
$$\veps X\cdot B^\H(X,X)+\veps Y\cdot B^\H(Y,Y)=0.$$
\end{proposition}

\begin{proof}
The foliation $\F$ is conformal if and only if there exists  a vector field $V\in\V$ such that $B^\H(E,F)=g(E,F)\otimes V$ for all $E,F\in\H$.  Let $\{X,Y\}$ be a local orthonormal frame for $\H$ and $\veps X$, $\veps Y$ the corresponding causal charachters.  Then 
$$\veps X\cdot B^\H(X,X)+\veps Y\cdot B^\H(Y,Y)=(\veps X\cdot g(X,X)+\veps Y\cdot g(Y,Y))\otimes V=2\,V$$
and hence $B^\H$ has the following form
\begin{equation}\label{equation-form}
B^\H(E,F)=\tfrac 12\cdot g(E,F)\otimes
\bigl(\veps X\cdot B^\H (X,X)+\veps Y\cdot B^\H(Y,Y)\bigr).
\end{equation}
This tells us that the conformality of $\F$ implies $B^\H(X,Y)=0$ and further
\begin{eqnarray*}
& &\veps X\cdot B^\H (X,X)-\veps Y\cdot B^\H(Y,Y)\\
&=&\tfrac 12\cdot (\veps X\cdot(g(X,X)-\veps Y\cdot g(Y,Y))\otimes V\\
&=&\tfrac 12\cdot (\veps X^2-\veps Y^2)\otimes V\\
&=&0.
\end{eqnarray*}
The other direction of the first statement is an easy exercise left to the reader. The second statement is an immediate consequence of formula (\ref{equation-form}).
\end{proof}

\section{Lie Foliations of Codimension Two}
\label{section-Lie-foliations}

Let $(G,g)$ be a semi-Riemannian Lie group with subgroup $K$ of codimension two and $\F$ be the Lie foliation on $G$ generated by the left-invariant translations of $K$.  Let $\V$ be the integrable distribution, tangent to the fibres of $\F$ and $\H$ be its complementary orthogonal distribution.  Further let $\{ V_1,\dots, V_n,X,Y\}$ be an orthonormal basis for the Lie algebra $\g$ such that $X$ and $Y$ generate $\H$.
\vskip .1cm

For the rest of this paper, we shall make extensive use of the well-known Koszul formula, adapted to the particular situation of semi-Riemannian Lie groups.  For left-invariant vector fields $X,Y,Z\in\g$ we have 
$$2\cdot g(\nab XY,Z)=g(\lb ZX,Y)+g(\lb ZY,X)+g(Z,\lb XY).$$
We also add the following standard calculations which turn out to be useful as we go along.

For the second fundamental form $B^\H:\H\times\H\to\V$, of the horizontal distribution $\H$, we see that for all left-invariant vector fields $E,F\in\H$
\begin{eqnarray*}
	B^\H(E,F)
	&=&\tfrac 12\cdot \bigl(\V\nabla_EF+\V\nabla_FE\bigr)\\
	&=&\tfrac 12\sum_{k=1}^n\varepsilon_{V_k}\cdot\bigl( g(\nabla_EF,V_k)+g(\nabla_FE,V_k)\bigr)V_k\\
	&=&\tfrac 12\sum_{k=1}^n\varepsilon_{V_k}\cdot\bigl(g(\lb E{V_k},F)
	+g(\lb F{V_k},E)\bigr)V_k.
\end{eqnarray*}
Similarly, we note that for the second fundamental form $B^\V:\V\times\V\to\H$, of the vertical distribution $\V$, we have for all left-invariant vector fields $E,F\in\V$
\begin{eqnarray*}
B^\V(E,F)
&=&\tfrac 12\cdot \bigl(\H\nabla_EF+\H\nabla_FE\bigr)\\
&=&\tfrac 12\cdot \bigl(\veps{X}\cdot (g(\nabla_EF,X)+g(\nabla_FE,X))X\\
& &\quad\quad +\,\veps{Y}\cdot (g(\nabla_EF,Y)+g(\nabla_FE,Y))Y\bigr)\\
&=&\tfrac 12\cdot \bigl(\varepsilon_{X} \cdot\bigl(g(\lb{X}{E},F)
+g(\lb{X}{F},E)\bigr)X\\
& &\quad\quad +\,\varepsilon_{Y}\cdot\bigl(g(\lb{Y}{E},F) +g(\lb{Y}{F},E)\bigr)Y\bigr).
\end{eqnarray*}

For the completeness of our exposition, we now state the following result.  This has already been applied in \cite{Gud-12} and \cite{Gud-Sve-6}, where the proving arguments can be found.

\begin{theorem}\label{theorem-simple-conformality}
Let $(G,g)$ be a semi-Riemannian Lie group with a subgroup $K$ generating a left-invariant conformal foliation $\F$ on $G$.  Let $\V$ be the integrable distribution tangent to $\F$ and $\H$ be the orthogonal complementary distribution of dimension two.  Then 
$$\H [\,[\V,\V],\H\,]=0.$$
\end{theorem}

\begin{corollary}\label{corollary-simple-conformality}
Let $(G,g)$ be a semi-Riemannian Lie group with a semisimple subgroup $K$, of codimension two, generating a left-invariant conformal foliation $\F$ on $G$.  Then the foliation $\F$ is semi-Riemannian.
\end{corollary}

\begin{proof}
Let $\V$ be the integrable distribution tangent to $\F$ and $\H$ be the orthogonal complementary distribution of dimension two.  Since the subgroup $K$ is semisimple we know that $[\V,\V]=\V$ and it then follows from Theorem 
\ref{theorem-simple-conformality} that $\H [\V,\H]=0$.  This proves the statement.
\end{proof}

\section{Five dimensional Lie Groups Foliated by $\SU 2$}
\label{section-SU2}

Let $(G,g)$ be a five dimensional semi-Riemannian Lie group with the {\it compact} subgroup $\SU 2$ generating a left-invariant conformal foliation $\F$ on $G$. Let $\g=\su 2\oplus\m$ be an orthogonal decomposition of the Lie algebra $\g$ of $G$ and let $\{A, B, C, X, Y\}$ be an orthonormal basis for $\g$ such that $A,B$ and $C$ generate the subalgebra $\su 2$. Then the Lie bracket relations for $\su 2$ are given by
$$[A, B] = 2\,C, \quad
[C, A] = 2\,B, \quad
[B, C] = 2\,A.$$
Since the Lie group $\SU 2$ is simple and the foliation $\F$ is conformal, we get from Theorem \ref{theorem-simple-conformality} that $\H[V,H] = 0$, for all $V \in \V, H \in \H$. The remaining bracket relations for $\g$ are thus given by
\begin{eqnarray*}
&\lb AX = a_{11}A + a_{12}B + a_{13}C,\quad
\lb AY = a_{21}A + a_{22}B + a_{23}C,\\
&\lb BX = b_{11}A + b_{12}B + b_{13}C,\quad
\lb BY = b_{21}A + b_{22}B + b_{23}C,\\
&\lb CX = c_{11}A + c_{12}B + c_{13}C,\quad
\lb CY = c_{21}A + c_{22}B + c_{23}C,\\
&\lb XY = \rho X + \theta_{1} A + \theta_{2} B + \theta_{3} C,
\end{eqnarray*}
for some, not necessarily independent, constant coefficients. By using the Jacobi identity we can now simplify this system.

\begin{proposition}\label{proposition-SU2}
Let $(G,g)$ be a five dimensional semi-Riemannian Lie group with the subgroup $\SU 2$, generating a left-invariant conformal foliation $\F$ on $G$. Let $\g=\su 2\oplus\m$ be an orthogonal decomposition of the Lie algebra $\g$ of $G$ and let $\{A, B, C, X, Y\}$ be an orthonormal basis for $\g$ such that $A,B$ and $C$ generate the subalgebra $\su 2$. Then the Lie bracket relations for $\g$ are given by
\begin{eqnarray*}
&\lb AB=2C, \quad\lb CA=2B, \quad\lb BC= 2A\\
&\lb AX=-b_{11}B - c_{11}C, \quad\lb AY=-b_{21}B - c_{21}C,\\
&\lb BX= b_{11}A - c_{12}C, \quad\lb BY= b_{21}A - c_{22}C,\\
&\lb CX= c_{11}A + c_{12}B, \quad\lb CY= c_{21}A + c_{22}B,\\
&\lb XY= \rho X + \theta_{1} A + \theta_{2} B + \theta_{3} C,
\end{eqnarray*}
where the real coefficients $b_{11}, b_{21}, c_{11}, c_{12}, c_{21}, c_{22}, \rho$ are arbitrary and $\theta$ satisfies 
	\begin{equation*}
		\begin{pmatrix}
			\theta_{1}\\
			\theta_{2}\\
			\theta_{3}
		\end{pmatrix}
		= \frac{1}{2}
		\begin{pmatrix}
			-\rho c_{12} + b_{11}c_{21} - b_{21}c_{11}\\
			\phantom{-}\rho c_{11} + b_{11}c_{22} - b_{21}c_{12}\\
			-\rho b_{11} + c_{11}c_{22} - c_{21}c_{12}
		\end{pmatrix}.
	\end{equation*}
\end{proposition}

\begin{proof}
We first consider the Jacobi identities involving the left-invariant vector fields $A, B, C, X$ and yield
\begin{eqnarray*}
0
&=&\jacobi{A}{B}{X}\\
& &\qquad 2\,((a_{13} + c_{11})A + (b_{13} + c_{12})B - (a_{11} + b_{12} - c_{13})C),\\
0&=&\jacobi{A}{C}{X}\\
& &\qquad 2\,(-(a_{12} + b_{11})A + (a_{11} - b_{12} + c_{13})B - (b_{13} + c_{12})C),\\
0&=&\jacobi{B}{C}{X}\\
& &\qquad 2\,((a_{11} - b_{12} - c_{13})A + (a_{12} + b_{11})B + (a_{13} + c_{11})C).
\end{eqnarray*}
Since these all have to vanish, in order for $\g$ to be a Lie algebra, we obtain the following system of equations
	\begin{equation*}
		\begin{pmatrix}
			a_{11} + b_{12} - c_{13}\\
			a_{11} - b_{12} + c_{13}\\
			a_{11} - b_{12} - c_{13}\\
			a_{12} + b_{11}\\
			a_{13} + c_{11}\\
			b_{13} + c_{12}
		\end{pmatrix}
		= 0.
	\end{equation*}
This simplifies to
	\begin{equation*}
		\begin{pmatrix}
			a_{12}\\
			a_{13}\\
			b_{13}
		\end{pmatrix}
		= -
		\begin{pmatrix}
			b_{11}\\
			c_{11}\\
			c_{12}
		\end{pmatrix},
		\quad
		\begin{pmatrix}
			a_{11}\\
			b_{12}\\
			c_{13}
		\end{pmatrix}
		= 0.
	\end{equation*}
Since the Lie bracket relations are symmetric in $X$ and $Y$, with respect to $A, B$ and $C$, we similarly obtain 
	\begin{equation*}
		\begin{pmatrix}
			a_{22}\\
			a_{23}\\
			b_{23}
		\end{pmatrix}
		= -
		\begin{pmatrix}
			b_{21}\\
			c_{21}\\
			c_{22}
		\end{pmatrix},
		\quad
		\begin{pmatrix}
			a_{21}\\
			b_{22}\\
			c_{23}
		\end{pmatrix}
		= 0.
	\end{equation*}

Next we consider the Jacobi identities involving both $X$ and $Y$, while keeping in mind the knowledge gained so far, and obtain
\begin{eqnarray*}
&\jacobi{A}{X}{Y} = (\rho b_{11} - c_{11}c_{22} + c_{21}c_{12} + 2\theta_{3})B\\
&\,+ (\rho c_{11} + b_{11}c_{22} - b_{21}c_{12} - 2\theta_{2})C,\\
&\jacobi{B}{X}{Y} = -(\rho b_{11} - c_{11}c_{22} + c_{21}c_{12} - 2\theta_{3})A\\
&\,+ (\rho c_{12} - b_{11}c_{21} + b_{21}c_{11} + 2\theta_{1})C,\\
&\jacobi{C}{X}{Y} = -(\rho c_{11} + b_{11}c_{22} - b_{21}c_{12} - 2\theta_{2})A\\
&\,- (\rho c_{12} - b_{11}c_{21} + b_{21}c_{11} + 2\theta_{1})B.
\end{eqnarray*}
This leads to the following system of equations
\begin{equation*}
\begin{pmatrix}
\rho c_{12} - b_{11}c_{21} + b_{21}c_{11} + 2\theta_{1}\\
\rho c_{11} + b_{11}c_{22} - b_{21}c_{12} - 2\theta_{2}\\
\rho b_{11} - c_{11}c_{22} + c_{21}c_{12} + 2\theta_{3}
\end{pmatrix}
= 0.
\end{equation*}
Solving for $(\theta_{1}, \theta_{2}, \theta_{3})$ we then yield
	\begin{equation*}
		\begin{pmatrix}
			\theta_{1}\\
			\theta_{2}\\
			\theta_{3}
		\end{pmatrix}
		= \frac{1}{2}
		\begin{pmatrix}
			-\rho c_{12} + b_{11}c_{21} - b_{21}c_{11}\\
			\phantom{-}\rho c_{11} + b_{11}c_{22} - b_{21}c_{12}\\
			-\rho b_{11} + c_{11}c_{22} - c_{21}c_{12}
		\end{pmatrix}.
	\end{equation*}
Following these calculations we can thus remove or equate some of the constants to get the simplified relations of the statement.
\end{proof}

With the following result we give a complete answer to when the conformal foliation $\F$ is minimal and even totally geodesic.  This shows that the general situation here is {\it different} from the special Riemannian case presented in Theorem 4.1 of \cite{Gud-12}.  Furthermore it shows that Conjecture \ref{conjecture-Riemannian} does not hold in the semi-Riemannian situation.

\begin{theorem}\label{theorem-SU2}
Let $(G,g)$ be a five dimensional semi-Riemannian Lie group with the subgroup $\SU 2$ generating a left-invariant conformal foliation $\F$ on $G$. Let $\g=\su 2\oplus\m$ be an orthogonal decomposition of the Lie algebra $\g$ of $G$ and let $\{A, B, C, X, Y\}$ be an orthonormal basis for $\g$ such that $A,B$ and $C$ generate the subalgebra $\su 2$. Then the foliation $\F$ is semi-Riemannian and minimal. It is totally geodesic if and only if the following conditions hold
	\begin{eqnarray*}
	0&=&(\veps B-\veps A)\,b_{11}=(\veps B-\veps A)\,b_{21}=(\veps C-\veps A)\,c_{11},\\
	0&=&(\veps C-\veps A)\,c_{21}=(\veps C-\veps B)\,c_{12}=(\veps C-\veps B)\,c_{22}.
\end{eqnarray*}
\end{theorem}

\begin{proof}
	The foliation $\F$ is minimal if and only if the trace of the second fundamental form $B^\V$ vanishes. We now have 
	\begin{eqnarray*}
		B^\V(A, A) 
		&=& \veps X\cdot g([X, A], A)X + \veps Y\cdot g([Y, A], A)Y\\
		&=& \veps X\cdot g(b_{11}B + c_{11}C, A)X + \veps Y\cdot g(b_{21}B + c_{21}C, A)Y\\ 
		&=& 0,
	\end{eqnarray*}
	\begin{eqnarray*}
		B^\V(B, B) &= & \veps X\cdot g([X, B], B)X + \veps Y\cdot g([Y, B], B)Y = 0,\\
		B^\V(C, C) &= & \veps X\cdot g([X, C], C)X + \veps Y\cdot g([Y, C], C)Y = 0.
	\end{eqnarray*}
Consequently, we see that $\F$ is minimal by default, independent of the choice of the metric $g$ or the structure coefficients of the Lie algebra. The following confirms the last statement.
\begin{eqnarray*}
2\,B^\V(A, B) &=& \veps X\cdot (g([X, A], B) + g([X, B], A))X\\
& & +\, \veps Y\cdot (g([Y, A], B) + g([Y, B], A))Y\\
&=& \veps X\cdot (g(b_{11}B + c_{11}C, B) + g(-b_{11}A + c_{12}C, A))X\\
& & +\, \veps Y\cdot (g(b_{21}B + c_{21}C, B) + g(-b_{21}A + c_{22}C, A))Y\\
&=& \veps X\cdot (b_{11}\veps B - b_{11}\veps A)X +\veps Y\cdot (b_{21}\veps B - b_{21}\veps A)Y,\\ \\
2\,B^\V(A, C) 
&=& \veps X\cdot (c_{11}\veps C - c_{11}\veps A)X + \veps Y\cdot (c_{21}\veps C - c_{21}\veps A)Y,\\
2\,B^\V(B, C) 
&=& \veps X\cdot (c_{12}\veps C - c_{12}\veps B)X + \veps Y\cdot(c_{22}\veps C - c_{22}\veps B)Y.
\end{eqnarray*}
\end{proof}

\begin{remark}
We would like to point out that the causal characters $\veps X$ and $\veps Y$ do not play any role for $\F$ being minimal or totally geodesic. If the metric $g$ is Riemannian then the foliation $\F$ is automatically totally geodesic. This was already stated in Theorem 4.1 of \cite{Gud-12}.  By analysing the conditions it is not difficult to see that we have obtained 
\begin{enumerate}
\item[(a)] one 7-dimensional family, where  $b_{11},b_{21},c_{11},c_{12},c_{21},c_{22},\rho$ are free.
\item[(b)] three 3-dimensional families, where the free variables are $b_{11}, b_{21}, \rho$ and $c_{11}, c_{21}, \rho$ and $c_{12}, c_{22}, \rho$, respectively.
\item[(c)] one 1-dimensional family, where $\rho$ is the only free variable. For each $\rho$ the Lie group $G$ is then a direct product of $\SU 2$ and a complete surface of constant curvature diffeomorphic to the plane.
\end{enumerate}
\end{remark}

\section{Five dimensional Lie Groups Foliated by $\SLR 2$}
\label{section-SLR2}

Let $(G,g)$ be a five dimensional semi-Riemannian Lie group with the {\it non-compact} subgroup $\SLR 2$ generating a left-invariant conformal foliation $\F$ on $G$. Let $\g=\slr 2\oplus\m$ be an orthogonal decomposition of the Lie algebra $\g$ of $G$ and let $\{A, B, C, X, Y\}$ be an orthonormal basis for $\g$ such that $A,B$ and $C$ generate the subalgebra $\slr 2$. Then the Lie bracket relations for $\slr 2$ are given by
$$[A, B] = 2\,C, \quad
[C, A] = 2\,B, \quad
[B, C] = -2\,A.$$
For this situation we have the following result.  This is similar to, but {\it different} from, Proposition \ref{proposition-SU2}. 

\begin{proposition}\label{proposition-SLR2}
Let $(G,g)$ be a five dimensional semi-Riemannian Lie group with the subgroup $\SLR 2$ generating a left-invariant conformal foliation $\F$ on $G$. Let $\g=\slr 2\oplus\m$ be an orthogonal decomposition of the Lie algebra $\g$ of $G$ and let $\{A, B, C, X, Y\}$ be an orthonormal basis for $\g$ such that $A,B$ and $C$ generate the subalgebra $\slr 2$. Then the Lie bracket relations for $\g$ are given by
\begin{eqnarray*}
&[A, B] =  2\,C, \quad [C, A] =  2\,B, \quad [B, C] = -2\,A, \\
&[A, X] =  b_{11}B + c_{11}C, \quad [A, Y] =  b_{21}B + c_{21}C, \\
&[B, X] =  b_{11}A - c_{12}C, \quad [B, Y] =  b_{21}A -c_{22}C,\\ 
&[C, X] =  c_{11}A + c_{12}B, \quad[C, Y] =  c_{21}A + c_{22}B, \\
&[X, Y] =  \rho X + \theta_{1}A + \theta_{2}B + \theta_{3}C,
\end{eqnarray*}
where the real coefficients $b_{11}, b_{21}, c_{11}, c_{12}, c_{21}, c_{22}, \rho$ are arbitrary and $\theta$ satisfies 
	\begin{equation*}
		\begin{pmatrix}
			\theta_{1} \\
			\theta_{2} \\
			\theta_{3}
		\end{pmatrix}
		= \frac{1}{2}
		\begin{pmatrix}
			-\rho c_{12} - c_{21}b_{11} + c_{11}b_{21}\\
			-\rho c_{11} - c_{22}b_{11} + c_{12}b_{21}\\
			\phantom{-}\rho b_{11} - c_{22}c_{11} + c_{12}c_{21}
		\end{pmatrix}.
	\end{equation*}
\end{proposition}

\begin{proof}
The result can be proven by exactly the same technique already presented for  Proposition \ref{proposition-SU2}.
\end{proof}

With the following result we give a complete answer to when the conformal foliation $\F$ is minimal and even totally geodesic.  Here the situation is {\it different} from that in Theorem \ref{theorem-SU2}, when the subgroup is the compact $\SU 2$.  It is worth noting that here the foliation is {\it not} automatically totally geodesic in the special Riemannian case.

\begin{theorem}\label{theorem-SLR2}
Let $(G,g)$ be a five dimensional semi-Riemannian Lie group with the subgroup $\SLR 2$ generating a left-invariant conformal foliation $\F$ on $G$. Let $\g=\slr 2\oplus\m$ be an orthogonal decomposition of the Lie algebra $\g$ of $G$ and let $\{A, B, C, X, Y\}$ be an orthonormal basis for $\g$ such that $A,B$ and $C$ generate the subalgebra $\slr 2$. Then the foliation $\F$ is semi-Riemannian and minimal.  It is totally geodesic if and only if 
\begin{eqnarray*}	
0&=&(\veps B+\veps A)\,b_{11}=(\veps B-\veps A)\,b_{21}=(\veps C+\veps A)\,c_{11},\\
0&=&(\veps C-\veps A)\,c_{21}=(\veps C+\veps B)\,c_{12}=(\veps C-\veps B)\,c_{22}.
\end{eqnarray*}
\end{theorem}

\begin{proof}
This result can be proven by the same method as Theorem \ref{theorem-SU2}.
\end{proof}

\section{Eight dimensional Lie Groups Foliated by $\SU 2\times\SU 2$}
\label{section-SU2SU2}

Let $(G, g)$ be an eight dimensional semi-Riemannian Lie group with the {\it compact} subgroup $K = \SU 2 \times \SU 2$ generating a left-invariant conformal foliation $\F$ on $G$. Let $\g=\k\oplus\m$ be an orthogonal decomposition of the Lie algebra $\g$ of $G$ and let $\{A, B, C, R, S, T, X, Y\}$ be an orthonormal basis for $\g$ such that the Lie subalgebra $\k = \su 2 \times \su 2$ is generated by the vector fields $A, B, C\in\su 2$ and $R, S, T\in\su 2$. Then the Lie bracket relations for $\k$ are given by
$$[A, B] = 2\,C, \quad
[C, A] = 2\,B, \quad
[B, C] = 2\,A.$$
$$[R, S] = 2\,T, \quad
[T, R] = 2\,S, \quad
[S, T] = 2\,R.$$
Since the subgroup $K$ is semisimple and the foliation $\F$ is confomal the additional Lie bracket relations are of the form
\begin{eqnarray*}
\lb AX &=& a_{11}A + a_{12}B + a_{13}C + a_{14}R + a_{15}S + a_{16}T,\\
\lb AY &=& a_{21}A + a_{22}B + a_{23}C + a_{24}R + a_{25}S + a_{26}T,\\
\lb BX &=& b_{11}A + b_{12}B + b_{13}C + b_{14}R + b_{15}S + b_{16}T,\\
\lb BY &=& b_{21}A + b_{22}B + b_{23}C + b_{24}R + b_{25}S + b_{26}T,\\
\lb CX &=& c_{11}A + c_{12}B + c_{13}C + c_{14}R + c_{15}S + c_{16}T,\\
\lb CY &=& c_{21}A + c_{22}B + c_{23}C + c_{24}R + c_{25}S + c_{26}T,\\
\lb RX &=& r_{11}A + r_{12}B + r_{13}C + r_{14}R + r_{15}S + r_{16}T,\\
\lb RY &=& r_{21}A + r_{22}B + r_{23}C + r_{24}R + r_{25}S + r_{26}T,\\
\lb SX &=& s_{11}A + s_{12}B + s_{13}C + s_{14}R + s_{15}S + s_{16}T,\\
\lb SY &=& s_{21}A + s_{22}B + s_{23}C + s_{24}R + s_{25}S + s_{26}T,\\
\lb TX &=& t_{11}A + t_{12}B + t_{13}C + t_{14}R + t_{15}S + t_{16}T,\\
\lb TY &=& t_{21}A + t_{22}B + t_{23}C + t_{24}R + t_{25}S + t_{26}T,\\
\lb XY &=& \rho X + \theta_{1} A + \theta_{2} B + \theta_{3} C + \theta_{4} R + \theta_{5} S + \theta_{6} T.
\end{eqnarray*}
This system of Lie bracket relations can now be simplified by evoking the Jabobi identity.

\begin{proposition}\label{proposition-SU2SU2}
Let $(G, g)$ be an eight dimensional semi-Riemannian Lie group with the subgroup $K = \SU 2 \times \SU 2$ generating a left-invariant conformal foliation $\F$ on $G$.  Let $\g=\k\oplus\m$ be an orthogonal decomposition of the Lie algebra $\g$ of $G$ and let $\{A, B, C, R, S, T, X, Y\}$ be an orthonormal basis for $\g$ such that the Lie subalgebra $\k = \su 2 \times \su 2$ is generated by the vector fields $A, B, C\in\su 2$ and $R, S, T\in\su 2$. Then the Lie bracket relations for $\g$ can be written as
\begin{eqnarray*}
&[A, B] = 2C, \quad[C, A] = 2B, \quad[B, C] = 2A, \\
&[R, S] = 2T, \quad[T, R] = 2S, \quad[S, T] = 2R, \\
&[A, X] = -b_{11}B - c_{11}C, \quad [A, Y] = -b_{21}B - c_{21}C,\\
&[B, X] =  b_{11}A - c_{12}C, \quad [B, Y] = b_{21}A - c_{22}C,\\
&[C, X] =  c_{11}A + c_{12}B, \quad [C, Y] = c_{21}A + c_{22}B,\\
&[R, X] = -s_{14}S - t_{14}T, \quad [R, Y] = -s_{24}S - t_{24}T,\\
&[S, X] =  s_{14}R - t_{15}T, \quad [S, Y] = s_{24}R - t_{25}T,\\
&[T, X] =  t_{14}R + t_{15}S, \quad [T, Y] = t_{24}R + t_{25}S,\\
&[X, Y] =  \rho X + \theta_{1} A + \theta_{2} B + \theta_{3} C + \theta_{4} R + \theta_{5} S + \theta_{6} T,
\end{eqnarray*}
where the real coefficients $b_{11}, b_{21}, c_{11}, c_{12}, c_{21}, c_{22}, s_{14}, s_{23}, t_{14}, t_{15}, t_{24}, t_{25}, \rho$ are arbitrary and $\theta$ satisfies 
	\begin{equation*}
		\begin{pmatrix}
			\theta_{1}\\
			\theta_{2}\\
			\theta_{3}\\
			\theta_{4}\\
			\theta_{5}\\
			\theta_{6}
		\end{pmatrix}
		=\frac{1}{2}
		\begin{pmatrix}
			-\rho c_{12} + b_{11}c_{21} - b_{21}c_{11}\\
			\phantom{-}\rho c_{11} + b_{11}c_{22} - b_{21}c_{12}\\
			-\rho b_{11} + c_{11}c_{22} - c_{12}c_{21}\\
			-\rho t_{15} + s_{14}t_{24} - s_{24}t_{14}\\
			\phantom{-}\rho t_{14} + s_{14}t_{25} - s_{24}t_{15}\\
			-\rho s_{14} + t_{14}t_{25} - t_{15}t_{24}
		\end{pmatrix}.
	\end{equation*}
\end{proposition}

\begin{proof}
Applying the Jacobi identity and the fact that $\lb AR=0$ we see that
$$0=\lb{\lb AR}X=\lb{\lb AX}R-\lb{\lb RX}A$$
Hence
\begin{eqnarray*}
0&=&\tfrac 12\,\lb{\lb AR}X=-r_{13}\,B+r_{12}\,C+a_{16}\,S-a_{15}\,T,\\
0&=&\tfrac 12\,\lb{\lb AS}X=-s_{13}\,B+s_{12}\,C-a_{16}\,R+a_{14}\,T,\\
0&=&\tfrac 12\,\lb{\lb AT}X=-t_{13}\,B+t_{12}\,C+a_{15}\,R-a_{14}\,S.
\end{eqnarray*}
Now replacing $X$ with $Y$, we directly obtain 
$$a_{24}=a_{25}=a_{26}=r_{22}=r_{23}=s_{22}=s_{23}=t_{22}=t_{23}=0.$$
Repeating the same procedure with $B$ and $C$ instead of $A$ provides the stated result.
\end{proof}

\begin{theorem}
Let $(G,g)$ be a semi-Riemannian Lie group with a semisimple product subgroup $K=K_1\times\cdots\times K_n$ generating a left-invariant conformal foliation $\F$ on $G$.  Let $\g=\k_1\oplus\cdots\oplus\k_n\oplus\m$ be an orthogonal decomposition of the Lie algebra $\g$ of $G$ and $\m$ be two dimensional.  For each $k$, let $\V_k$ be the integrable distribution generated by $\k_k$, $\V=\V_1\oplus\cdots\oplus V_n$ and $\H$ be its orthogonal complementary distribution generated by $\m$.  Then for all $j\neq k$ we have 
$$\V_j[\V_k,\H]=0.$$
\end{theorem}

\begin{proof}
The result can be proven by exactly the same method applied for Proposition \ref{proposition-SU2SU2}.
\end{proof}

With the next theorem we give a complete answer to when the conformal foliation $\F$ is minimal and even totally geodesic.   This shows that the general situation here is {\it different} from the special Riemannian case presented in Theorem 4.2 of \cite{Gha-Gud-Tur-1}.  Furthermore it shows again that Conjecture \ref{conjecture-Riemannian} does not hold in the semi-Riemannian situation.

\begin{theorem}\label{theorem-SU2SU2}
Let $(G, g)$ be an eight dimensional semi-Riemannian Lie group with the subgroup  $K = \SU 2 \times \SU 2$, generating a left-invariant conformal foliation $\F$ on $G$. Let $\g = \k \oplus \m$ be an orthogonal decomposition of the Lie algebra of $G$ such that $\k = \su 2 \times \su 2$. Furthermore, let $\{A, B, C, T, X, Y\}$ be an orthonormal basis for $\g$ such that the elements $A, B, C\in\su 2$ and $R,S,T\in\su 2$. Then  the foliation $\F$ is semi-Riemannian and minimal. It is totally geodesic if and only if the following conditions hold
\begin{eqnarray*}
0&=&(\veps B-\veps A)\,b_{11}=(\veps B-\veps A)\,b_{21}=(\veps C-\veps A)\,c_{11},\\
0&=&(\veps C-\veps A)\,c_{21}=(\veps C-\veps B)\,c_{12}=(\veps C-\veps B)\,c_{22},\\
0&=&(\veps S-\veps R)\,s_{11}=(\veps S-\veps R)\,s_{21}=(\veps T-\veps R)\,t_{11},\\
0&=&(\veps T-\veps R)\,t_{21}=(\veps T-\veps S)\,t_{12}=(\veps T-\veps S)\,t_{22}.
	\end{eqnarray*}
\end{theorem}
\begin{proof}
	The fact that the foliation $\F$ is semi-Riemannian already follows from Proposition \ref{proposition-SU2SU2} and the minimality from an elementary calculations showing that for all $E\in\{A,B,C,R,S,T\}$ we have $B^\V(E, E)=0$.    A similar computation yields 
	\begin{eqnarray*}
		B^\V(A, B)&=&\tfrac{1}{2}\ (\veps{X}\cdot(b_{11}\veps B - b_{11}\veps A)X + \veps{Y}\cdot(b_{21}\veps B - b_{21}\veps A)Y),\\
		B^\V(A, C)&=&\tfrac{1}{2} (\veps{X}\cdot(c_{11}\veps C - c_{11}\veps A)X + \veps{Y}\cdot(c_{21}\veps C - c_{21}\veps A)Y),\\
		B^\V(B, C)&=&\tfrac{1}{2} (\veps{X}\cdot(c_{12}\veps C - c_{12}\veps B)X + \veps{Y}\cdot(c_{22}\veps C - c_{22}\veps B)Y).
	\end{eqnarray*}
	This gives the first half of the conditions in the statement.  The second part is obtained in exactly the same way.
\end{proof}

The reader should note that, in the special Riemannian case, the twelve  conditions in Theorem \ref{theorem-SU2SU2} are trivially satisfied.  This shows that Conjecture  \ref{conjecture-Riemannian} holds in that particular case.

\section{Eight dimensional Lie Groups Foliated by $\SU 2\times\SLR 2$}
\label{section-SU2SLR2}

Let $(G, g)$ be an eight dimensional semi-Riemannian Lie group with the {\it  non-compact} subgroup $K = \SU 2 \times \SLR 2$ generating a left-invariant conformal foliation $\F$ on $G$. We can now apply the result of Theorem \ref{theorem-simple-conformality} and the Jacobi identity to get the following result on the structure of the Lie algebra $g$ of $G$.

\begin{proposition}\label{proposition-SU2SLR2}
Let $(G, g)$ be an eight dimensional semi-Riemannian Lie group with the  subgroup $K = \SU 2 \times \SLR 2$ generating a left-invariant conformal foliation $\F$ on $G$.  Let $\g = \k \oplus \m$ be an orthogonal decomposition of the Lie algebra of $G$ such that $\k = \su 2 \times \slr 2$.  Furthermore, let $\{A, B, C, R, S, T, X, Y\}$ be an orthonormal basis for $\g$ such that the Lie subalgebra $\k = \su 2 \times \slr 2$ is generated by the vector fields $A, B, C\in\su 2$ and $R, S, T\in\slr 2$.	Then the Lie bracket relations for $\g$ can be written as 
\begin{eqnarray*}
&[A, B] =  2C,\quad[C, A] =  2B,\quad[B, C] =  2A,\\
&[R, S] =  2T,\quad[T, R] =  2S,\quad[S, T] = -2R,\\
&[A, X] = -b_{11} B - c_{11} C, \quad[A, Y] = -b_{21} B - c_{21} C,\\
&[B, X] =  b_{11} A - c_{12} C, \quad[B, Y] =  b_{21} A - c_{22} C,\\
&[C, X] =  c_{11} A + c_{12} B, \quad[C, Y] =  c_{21} A + c_{22} B,\\
&[R, X] =  s_{14} S + t_{14} T, \quad[R, Y] =  s_{24} S + t_{24} T,\\
&[S, X] =  s_{14} R - t_{15} T, \quad[S, Y] =  s_{24} R - t_{25} T,\\
&[T, X] =  t_{14} R + t_{15} S, \quad[T, Y] =  t_{24} R + t_{25} S,\\
&[X, Y] =  \rho X + \theta_{1} A + \theta_{2} B + \theta_{3} C + \theta_{4} R + \theta_{5} S + \theta_{6} T,
\end{eqnarray*}
where the real coefficients $b_{11}, b_{21}, c_{11}, c_{12}, c_{21}, c_{22}, s_{14}, s_{23}, t_{14}, t_{15}, t_{24}, t_{25}, \rho$ are arbitrary and $\theta$ satisfies 
	\begin{equation*}
		\begin{pmatrix}
			\theta_{1}\\
			\theta_{2}\\
			\theta_{3}\\
			\theta_{4}\\
			\theta_{5}\\
			\theta_{6}
		\end{pmatrix}
		=\frac{1}{2}
		\begin{pmatrix}
			-\rho c_{12} + b_{11}c_{21} - b_{21}c_{11}\\
			\phantom{-}\rho c_{11} + b_{11}c_{22} - b_{21}c_{12}\\
			-\rho b_{11} + c_{11}c_{22} - c_{12}c_{21}\\
			-\rho t_{15} - s_{14}t_{24} + s_{24}t_{14}\\
			-\rho t_{14} - s_{14}t_{25} + s_{24}t_{15}\\
			\phantom{-}\rho s_{14} - t_{14}t_{25} + t_{15}t_{24}
		\end{pmatrix}.
	\end{equation*}
\end{proposition}

\begin{proof}
The statement can be proven by using exactly the same arguments as we did for Proposition \ref{proposition-SU2SU2}.
\end{proof}

\begin{theorem}\label{theorem-SU2SLR2}
Let $(G, g)$ be an eight dimensional semi-Riemannian Lie group with the  subgroup $K = \SU 2 \times \SLR 2$ generating a left-invariant conformal foliation $\F$ on $G$.  Let $\g = \k \oplus \m$ be an orthogonal decomposition of the Lie algebra of $G$ such that $\k = \su 2 \times \slr 2$.  Furthermore, let $\{A, B, C, R, S, T, X, Y\}$ be an orthonormal basis for $\g$ such that the Lie subalgebra $\k = \su 2 \times \slr 2$ is generated by the vector fields $A, B, C\in\su 2$ and $R, S, T\in\slr 2$.	 Then the foliation $\F$ is both semi-Riemannian and minimal.  Moreover, $\F$ is totally geodesic if and only if one of the following conditions holds.
\begin{eqnarray*}	
0&=&(\veps B-\veps A)\,b_{11}=(\veps B-\veps A)\,b_{21}=(\veps C-\veps A)\,c_{11},\\
0&=&(\veps C-\veps A)\,c_{21}=(\veps C-\veps B)\,c_{12}=(\veps C-\veps B)\,c_{22},\\
0&=&(\veps S+\veps R)\,s_{14}=(\veps S-\veps R)\,s_{24}=(\veps T+\veps R)\,t_{14},\\
0&=&(\veps T-\veps R)\,t_{24}=(\veps T+\veps S)\,t_{15}=(\veps T-\veps S)\,t_{25}.
\end{eqnarray*}
\end{theorem}

\begin{proof}
Here the technique is the same as that for Theorem \ref{theorem-SU2SLR2}.
\end{proof}

\section{Six dimensional Lie Groups Foliated by $\SU 2\times\SO 2$}
\label{section-SU2SO2}

Let $(G, g)$ be a six dimensional semi-Riemannian Lie group with the {\it compact} subgroup $K = \SU 2 \times \SO 2$ generating a left-invarienat Lie foliation $\F$ on $G$.  Here the situation is {\it different} from what we had earlier, since the subgroup $K$ is not semisimple. Let $\g = \k \oplus \m$ be an orthogonal decomposition of the Lie algebra of $G$ such that $\k = \su 2 \times \so 2$.  Furthermore, let $\{A,B,C,T,X,Y\}$ be an orthonormal basis for $\g$ such that $A,B,C$ generate $\su 2$ and $T$ the abelian $\so 2$. Then the Lie bracket relations for $\g$ are of the form
\begin{eqnarray*}
	&[A, B] = 2C, \quad[C, A] = 2B, \quad[B, C] = 2A, \\
	&[A, X] = a_{11}A + a_{12}B + a_{13}C + a_{14}T, \\
	&[A, Y] = a_{21}A + a_{22}B + a_{23}C + a_{24}T, \\
	&[B, X] = b_{11}A + b_{12}B + b_{13}C + b_{14}T, \\
	&[B, Y] = b_{21}A + b_{22}B + b_{23}C + b_{24}T, \\
	&[C, X] = c_{11}A + c_{12}B + c_{13}C + c_{14}T, \\
	&[C, Y] = c_{21}A + c_{22}B + c_{23}C + c_{24}T, \\
	&[T, X] =  x_{1}X +  y_{1}Y + t_{11}A + t_{12}B + t_{13}C + t_{14}T, \\
	&[T, Y] =  x_{2}X +  y_{2}Y + t_{21}A + t_{22}B + t_{23}C + t_{24}T, \\
	&[X, Y] = \rho X + \theta_{1}A + \theta_{2}B + \theta_{3}C + \theta_{4}T,
\end{eqnarray*}
By invoking the Jacobi identity we can simplify this system.  The next result gives an interesting criteria for when the foliation $\F$ is conformal and even semi-Riemannian.

\begin{lemma}
Let $(G, g)$ be a six dimensional semi-Riemannian Lie group with subgroup $K = \SU 2 \times \SO 2$. Then the foliation $\F$, tangent to the vertical distribution $\V$ generated by $K$, is conformal if and only if
$$x_{1}=y_{2}\text{ and }\veps{X}x_{2}+\veps{Y}y_{1}=0.$$
In that case, $\F$ is semi-Riemannian if and only if $x_1=y_2=0$.
\end{lemma}

\begin{proof}
	Employing the formula for the second fundamental form $B^\H$ of the  horizontal distribution we get 
	$$\veps XB^\H(X,X)-\veps YB^\H(Y,Y)=\veps T(x_1-y_2)\cdot T,$$
	$$B^\H(X,Y)=\tfrac 12\,\veps T(\veps Xx_2+\veps Yy_1)\cdot T.$$
	This proves the first statement.  The second claim follows from
	$$\veps XB^\H(X,X)+\veps YB^\H(Y,Y)=\veps T(x_1+y_2)\cdot T=0.$$
\end{proof}

\begin{theorem}\label{theorem-SU2SO2}
Let $(G,g)$ be a six dimensional semi-Riemannian Lie group with the subgroup $K=\SU 2\times\SO 2$ generating a left-invariant conformal foliation $\F$ on $G$. Let $\g = \k \oplus \m$ be an orthogonal decomposition of the Lie algebra of $G$ such that $\k = \su 2 \times \so 2$. Furthermore, let $\{A, B, C, T, X, Y\}$ be an orthonormal basis for $\g$ such that the elements $A, B, C$ generate $\su 2$ and $T$ the abelian $\so 2$. Then the structure of $\g$ is given by
\begin{eqnarray*}
&[A, B] = 2C, \quad[C, A] = 2B, \quad[B, C] = 2A, \\
&[A, X] = -b_{11}B - c_{11}C, \quad[A, Y] = -b_{21}B - c_{21}C, \\
&[B, X] =  b_{11}A - c_{12}C, \quad[B, Y] =  b_{21}A - c_{22}C, \\
&[C, X] =  c_{11}A + c_{12}B, \quad[C, Y] =  c_{21}A + c_{22}B, \\
&[T, X] = x_{1}X + y_{1}Y - \frac{1}{2}((x_{1}c_{12} + y_{1}c_{22})A \\
&\,- (x_{1}c_{11} + y_{1}c_{21})B + (x_{1}b_{11} + y_{1}b_{21})C) + t_{14}T,\\
&[T, Y] = x_{2}X + y_{2}Y - \frac{1}{2}((x_{2}c_{12} + y_{2}c_{22})A \\
&\,- (x_{2}c_{11} + y_{2}c_{21})B + (x_{2}b_{11} + y_{2}b_{21})C) + t_{24}T,\\
&[X, Y] =  \rho X + \theta_{1}A + \theta_{2}B + \theta_{3}C + \theta_{4}T,
\end{eqnarray*}	where the real coefficients $b_{11}, b_{21}, c_{11}, c_{12}, c_{21}, c_{22},t_{14}, t_{24}, \rho$ are arbitrary,
$$x_{1}=y_{2},\ \ \veps{X}x_{2}+\veps{Y}y_{1}=0$$
and $\theta$ satisfies 
\begin{equation*}
\begin{pmatrix}
\theta_{1} \\
\theta_{2} \\
\theta_{3} \\
\end{pmatrix}
= \frac{1}{2}
\begin{pmatrix}
-\rho c_{12} + b_{11}c_{21} - b_{21}c_{11} \\
\phantom{-}\rho c_{11} + b_{11}c_{22} - b_{21}c_{12} \\
-\rho b_{11} + c_{11}c_{22} - c_{12}c_{21}
\end{pmatrix}.
\end{equation*}
The foliation $\F$ is minimal if and only if $t_{14} = t_{24} = 0$ and in that case totally geodesic if and only if 
\begin{eqnarray*}
&0=b_{11}\,(\veps B-\veps A)=b_{21}\,(\veps B-\veps A),\\
&0=c_{11}\,(\veps C-\veps A)=c_{21}\,(\veps C-\veps A),\\
&0=c_{12}\,(\veps C-\veps B)=c_{22}\,(\veps C-\veps B),
\end{eqnarray*}
\begin{eqnarray*}
&0=(x_1c_{12}+y_1c_{22})=(x_2c_{12}+y_2c_{22}),\\
&0=(x_1c_{11}+y_1c_{21})=(x_2c_{11}+y_2c_{21}),\\
&0=(x_1b_{11}+y_1b_{21})=(x_2b_{11}+y_2b_{21}).
\end{eqnarray*}
\end{theorem}

\begin{proof}
Here we can apply exactly the same technique presented in the proof of Proposition \ref{proposition-SU2}.
\end{proof}

\section{Six dimensional Lie Groups Foliated by $\SLR 2\times\SO 2$}
\label{section-SLR2SO2}

Let $(G, g)$ be a six dimensional semi-Riemannian Lie group with the {\it non-compact} subgroup $K = \SLR 2 \times \SO 2$, generating a left-invarienat Lie foliation $\F$ on $G$.  Here the situation is {\it similar} to that in Section \ref{section-SU2SO2} in the sense that the subgroup $K$ is not semisimple. Let $\g = \k \oplus \m$ be an orthogonal decomposition of the Lie algebra of $G$ such that $\k = \slr 2 \times \so 2$.  Furthermore, let $\{A,B,C,T,X,Y\}$ be an orthonormal basis for $\g$ such that $A,B,C$ generate $\slr 2$ and $T$ the abelian $\so 2$. Then the Lie bracket relations for $\g$ are of the form
\begin{eqnarray*}
	&[A, B] = 2C, \quad[C, A] = 2B, \quad[B, C] = -2A, \\
	&[A, X] = a_{11}A + a_{12}B + a_{13}C + a_{14}T, \\
	&[A, Y] = a_{21}A + a_{22}B + a_{23}C + a_{24}T, \\
	&[B, X] = b_{11}A + b_{12}B + b_{13}C + b_{14}T, \\
	&[B, Y] = b_{21}A + b_{22}B + b_{23}C + b_{24}T, \\
	&[C, X] = c_{11}A + c_{12}B + c_{13}C + c_{14}T, \\
	&[C, Y] = c_{21}A + c_{22}B + c_{23}C + c_{24}T, \\
	&[T, X] =  x_{1}X +  y_{1}Y + t_{11}A + t_{12}B + t_{13}C + t_{14}T, \\
	&[T, Y] =  x_{2}X +  y_{2}Y + t_{21}A + t_{22}B + t_{23}C + t_{24}T, \\
	&[X, Y] = \rho X + \theta_{1}A + \theta_{2}B + \theta_{3}C + \theta_{4}T,
\end{eqnarray*}

For this situation we have the following result.

\begin{theorem}\label{theorem-SLR2SO2}
Let $(G,g)$ be a six dimensional semi-Riemannian Lie group with the subgroup $K=\SLR 2\times\SO 2$, generating a left-invariant conformal foliation $\F$ on $G$. Let $\g = \k \oplus \m$ be an orthogonal decomposition of the Lie algebra of $G$ such that $\k = \slr 2 \times \so 2$. Furthermore, let $\{A, B, C, T, X, Y\}$ be an orthonormal basis for $\g$ such that the elements $A, B, C$ generate $\slr 2$ and $T$ the abelian $\so 2$. Then the structure of $\g$ is given by
\begin{eqnarray*}
&[A, B] = 2C, \quad[C, A] = 2B, \quad[B, C] = -2A, \\
&[A, X] = b_{11}B + c_{11}C, \quad[A, Y] = b_{21}B + c_{21}C, \\
&[B, X] =  b_{11}A - c_{12}C, \quad[B, Y] =  b_{21}A - c_{22}C, \\
&[C, X] =  c_{11}A + c_{12}B, \quad[C, Y] =  c_{21}A + c_{22}B, \\
&[T, X] = x_{1}X + y_{1}Y - \frac{1}{2}((x_{1}c_{12} + y_{1}c_{22})A \\
&\,+ (x_{1}c_{11} + y_{1}c_{21})B - (x_{1}b_{11} + y_{1}b_{21})C) + t_{14}T,\\
&[T, Y] = x_{2}X + y_{2}Y - \frac{1}{2}((x_{2}c_{12} + y_{2}c_{22})A \\
&\,- (x_{2}c_{11} + y_{2}c_{21})B + (x_{2}b_{11} + y_{2}b_{21})C) + t_{24}T,\\
&[X, Y] =  \rho X + \theta_{1}A + \theta_{2}B + \theta_{3}C + \theta_{4}T,
\end{eqnarray*}	where the real coefficients $b_{11}, b_{21}, c_{11}, c_{12}, c_{21}, c_{22},t_{14}, t_{24}, \rho$ are arbitrary,
$$x_{1}=y_{2},\ \ \veps{X}x_{2}+\veps{Y}y_{1}=0$$
and $\theta$ satisfies 
\begin{equation*}
\begin{pmatrix}
\theta_{1} \\
\theta_{2} \\
\theta_{3} \\
\end{pmatrix}
= \frac{1}{2}
\begin{pmatrix}
-\rho c_{12} - b_{11}c_{21} + b_{21}c_{11} \\
-\rho c_{11} - b_{11}c_{22} + b_{21}c_{12} \\
\phantom{-} \rho b_{11} - c_{11}c_{22} + c_{12}c_{21}
\end{pmatrix}.
\end{equation*}

The foliation $\F$ is minimal if and only if $t_{14} = t_{24} = 0$ and in that case totally geodesic if and only if 
\begin{eqnarray*}
&0=b_{11}\,(\veps B+\veps A)=b_{21}\,(\veps B+\veps A),\\
&0=c_{11}\,(\veps C+\veps A)=c_{21}\,(\veps C+\veps A),\\
&0=c_{12}\,(\veps C-\veps B)=c_{22}\,(\veps C-\veps B),
\end{eqnarray*}
\begin{eqnarray*}
&0=(x_1c_{12}+y_1c_{22})=(x_2c_{12}+y_2c_{22}),\\
&0=(x_1c_{11}+y_1c_{21})=(x_2c_{11}+y_2c_{21}),\\
&0=(x_1b_{11}+y_1b_{21})=(x_2b_{11}+y_2b_{21}).
\end{eqnarray*}
\end{theorem}

\begin{proof}
The result is obtained in exactly the same way as that of Theorem \ref{theorem-SU2SO2}.
\end{proof}

\section{Acknowledgements}

The first author would like to thank the Department of Mathematics at Lund University for its great hospitality during her time there as a postdoc.

\end{document}